\documentclass[11pt]{article}
\usepackage{amsmath,amssymb,amsthm,color}
\usepackage{comment}

\newtheorem{theorem}{Theorem}[section]
\newtheorem{corollary}{Corollary}[section]
\newtheorem{lemma}{Lemma}[section]
\newtheorem{claim}{Claim}[section]

\newtheorem{remark}{Remark}[section]

\def\R{{\mathfrak R}\, }
\def\D{{\mathfrak D}\, }
\def\A{{\mathfrak A}\, }

\begin{document}
\begin{center}{\bf \LARGE $*$-Lie-Jordan-type maps on $C^*$-algebras}\\
\vspace{.2in}
\noindent {\bf Bruno Leonardo Macedo Ferreira}\\
{\it Technological Federal University of Paran\'{a},\\
Avenida Professora Laura Pacheco Bastos, 800,\\
85053-510, Guarapuava, Brazil.}\\
e-mail: brunoferreira@utfpr.edu.br
\\
and
\\
\noindent {\bf Bruno Tadeu Costa}\\
{\it Federal University of Santa Catarina,\\
Rua Jo\~{a}o Pessoa, 2750,\\   
89036-256, Blumenau, Brazil.}\\
e-mail: b.t.costa@ufsc.br

\end{center}
\begin{abstract} 
Let $\A$ and $\A'$ be two  
$C^*$-algebras with identities $I_{\A}$ and $I_{\A'}$, respectively, and $P_1$ and $P_2 = I_{\A} - P_1$ nontrivial projections in $\A$. In this paper we study
the  characterization of multiplicative $*$-Lie-Jordan-type maps, where the notion of these maps arise here. In particular, if $\mathcal{M}_{\A}$ is a von Neumann algebra relative $C^{*}$-algebra $\A$ without central summands of type $I_1$ then every bijective unital multiplicative $*$-Lie-Jordan-type maps are $*$-ring isomorphisms.
\end{abstract}

\section{Introduction and Preliminaries}

The structures of some Jordan and Lie maps on associative and \linebreak non-associative rings were and has been studied systematically by many people, such as 
\cite{Abdu, Benkovic, BenkovicEremita, Bresar, changd,  FerGur1, FerGur2, Fergur3, Fos, Mar64}. A basic question about maps of general algebras is: when are they additives? See, for instance, 
\cite{chang, changd, Fer, bruth, FerGur1, Mart}. In other words, we are interested in finding out which is the sharp condition for a map to be additive. With this picture in mind, we present
in this paper the definition of a multiplicative $*$-Lie-Jordan-type map.

Many mathematicians devoted themselves to the study of new products on rings. Between them, two of these products on a $*$-ring $\R$ will be of special interest to us: for 
$A, B \in \R$, define by $\left\{A,B\right\}_{*} = AB+BA^{*}$ and $[A, B]_{*} = AB-BA^{*}$ two different kinds of products. Just out of curiosity, the product $[A, B]_{*}$ 
was extensively studied because, by the fundamental theorem of $\check{S}$emrl in \cite{semrl1}, maps of the form $T \mapsto TA - AT^{*}$ naturally arise in the problem of representing quadratic functionals with sesquilinear functionals (see \cite{semrl2, semrl3}).

In view of this broad and active line of research, we shall address the structure of a new class of maps named multiplicative $*$-Lie-Jordan-type maps on $C^*$-algebras, also introducing two different new products that extend $\left\{A,B\right\}_{*}$ and $[A, B]_{*}$. Our main objective is to study the characterization of multiplicative $*$-Lie-Jordan-type maps on $C^*$-algebras. As a consequence, if $\mathcal{M}_{\A}$ is a von Neumann algebra relative $C^{*}$-algebra $\A$ without central summands of type $I_1$ then every bijective unital multiplicative $*$-Lie-Jordan-type maps are $*$-ring isomorphisms.

Let $\mathfrak{R}$ and $\mathfrak{R}'$ be two rings and $\varphi\colon\mathfrak{R}\longrightarrow\mathfrak{R}'$ be a map. 
We say $\varphi$ is \textit{additive} if $\varphi(a+b)=\varphi(a)+\varphi(b)$.

Let $\mathfrak{R}$ be a ring with centre $\mathcal{Z}(\mathfrak{R})$ and $\left[x_1,x_2\right]_{*} = x_1x_2-x_2x_1^{*}$ 
and $\left\{x_1, x_2\right\}_{*} = x_1x_2+x_2x_1^{*}$ denote the usual $*$-Lie product and $*$-Jordan product of $x_1$ and $x_2$, respectively, where $*$ is a involution.
Let us define the following sequence of polynomials: 
$$p_{1*}(x) = x\, \,  \text{and}\, \,  p_{n*}(x_1, x_2, \ldots , x_n) = [p_{(n-1)*}(x_1, x_2, \ldots , x_{n-1}) , x_n]_{*}$$
$$q_{1*}(x) = x\, \,  \text{and}\, \,  q_{n*}(x_1, x_2, \ldots , x_n) = \left\{q_{(n-1)*}(x_1, x_2, \ldots , x_{n-1}) , x_n\right\}_{*},$$
for all integers $n \geq 2$. 
Thus, $p_{2*}(x_1, x_2) = [x_1, x_2]_{*}, \  p_{3*} (x_1, x_2, x_3) = [[x_1, x_2]_{*} , x_3]_{*}, \ q_{2*}(x_1, x_2) = \left\{x_1, x_2\right\}_{*}, \ q_{3*} (x_1, x_2, x_3) = \left\{\left\{x_1, x_2\right\}_{*} , x_3\right\}_{*}$, etc. Note that $p_{2*}$ and $q_{2*}$ are product introduced by Bre$\check{s}$ar and Fo$\check{s}$ner \cite{brefos1, brefos2}. 
Now let us define a new class of maps (not necessarily additive): $\varphi : \R \longrightarrow \R'$ shall be called \textit{multiplicative $*$-Lie-Jordan $n$-map} if
\begin{eqnarray*}\label{ident1}
&&\varphi(p_{n*} (x_1, x_2, . . . ,x_n)) =  p_{n*} (\varphi(x_1), \varphi(x_2), . . . , \varphi(x_i), . . .,\varphi(x_n))
\\&&\varphi(q_{n*} (x_1, x_2, . . . ,x_n)) =  q_{n*} (\varphi(x_1), \varphi(x_2), . . . , \varphi(x_i), . . .,\varphi(x_n))
\end{eqnarray*}
where $n \geq 2$ is a integer.
Multiplicative $*$-Lie-Jordan $2$-map, $*$-Lie-Jordan $3$-map 
and $*$-Lie-Jordan $n$-map are collectively referred to as \textit{multiplicative $*$-Lie-Jordan-type maps}.
In the case $\varphi$ satisfies only
\begin{eqnarray*}\label{ident1}
&&\varphi(p_{n*} (x_1, x_2, . . . ,x_n)) =  p_{n*} (\varphi(x_1), \varphi(x_2), . . . , \varphi(x_i), . . .,\varphi(x_n)),
\end{eqnarray*}
we say that $\varphi$ is a \textit{multiplicative $*$-Lie $n$-map}. On the other hand, if $\varphi$ satisfies only
\begin{eqnarray*}\label{ident1}
\varphi(q_{n*} (x_1, x_2, . . . ,x_n)) =  q_{n*} (\varphi(x_1), \varphi(x_2), . . . , \varphi(x_i), . . .,\varphi(x_n)),
\end{eqnarray*}
we say that $\varphi$ is a \textit{multiplicative $*$-Jordan $n$-map}.
Also multiplicative $*$-Lie $2$-map, $*$-Lie $3$-map, $*$-Lie $n$-map and multiplicative $*$-Jordan $2$-map, $*$-Jordan $3$-map 
and $*$-Jordan $n$-map  are collectively referred to as \textit{multiplicative $*$-Lie-type maps} and \textit{multiplicative $*$-Jordan-type maps}, respectively. 

By a $C^*$-algebra we mean a complete normed complex algebra (say $\A$) endowed with a conjugate-linear algebra involution $*$, 
satisfying $||a^*a||=||a||^2$ for all $a \in \A$. Moreover, a $C^*$-algebra is a {\it prime} $C^*$-algebra if $A\A B = 0$ for 
$A,B \in \A$ implies either $A=0$ or $B=0$.

We find it convenient to invoke the noted Gelfand-Naimark theorem that state a $C^*$-algebra $\A$ is $*$-isomorphic to a
$C^*$-subalgebra $\D \subset \mathcal{B}(\mathcal{H})$, where $\mathcal{H}$ is a Hilbert space. So from now on we shall
consider elements of a $C^*$-algebra as operators.

Let $\mathbb{R}$ and $\mathbb{C}$ denote, respectively, the real field and complex field and for real part and imaginary
part of an operator $T$ we shall use $\Re(T)$ and $\Im(T)$, respectively. 

Let be $P_1$ a nontrivial projection in $\A$ and $P_2 = I_{\A} - P_1$ where $I_{\A}$ is the identity of $\A$. Then $\mathfrak{A}$ has a decomposition
$\mathfrak{A}=\mathfrak{A}_{11}\oplus \mathfrak{A}_{12}\oplus
\mathfrak{A}_{21}\oplus \mathfrak{A}_{22},$ where
$\mathfrak{A}_{ij}=P_{i}\mathfrak{A}P_{j}$ $(i,j=1,2)$.

\begin{remark}\label{obs}
By definition of involution clearly we get $*(\A_{ij}) \subseteq \A_{ji}$ for $i,j \in \left\{1,2\right\}.$
\end{remark}

\section{Multiplicative $*$-Lie-Jordan-type maps}

\subsection{Auxiliary Claims} 

\begin{claim}\label{Claim1} Let $\A$ and $\A'$ be two $C^*$-algebras with identities
and $\varphi$ be a unital multiplicative $*$-Lie-Jordan $n$-map of $\mathfrak{A}$ onto an arbitrary $C^*$-algebra $\mathfrak{A}'$. Then $\varphi(0) = 0$.
\end{claim}

\begin{proof}
Let $I_{\A}$ be the identity of $\A$. We have, 
$$\varphi(0) = \varphi(p_{n*} (I_{\A}, I_{\A}, . . . ,I_{\A})) =  p_{n*} (\varphi(I_{\A}), \varphi(I_{\A}), . . . ,\varphi(I_{\A}))= 0,$$
because $\varphi$ is unital. 
\end{proof}

\begin{claim}\label{Claim2} Let $\A$ and $\A'$ be two $C^*$-algebras and $\varphi$ be a multiplicative $*$-Lie-Jordan $n$-map of $\mathfrak{A}$ onto an arbitrary $C^*$-algebra $\mathfrak{A}'$. Let $A,B$ and $H$ be in $\A$ such that $\varphi(H) = \varphi(A) + \varphi(B)$. Then $$\varphi(p_{n*}(H, T_2, ..., T_n)) = \varphi(p_{n*}(A, T_2, ..., T_n)) + \varphi(p_{n*}(B, T_2, ..., T_n))$$ and $$\varphi(q_{n*}(H, T_2, ..., T_n)) = \varphi(q_{n*}(A, T_2, ..., T_n)) + \varphi(q_{n*}(B, T_2, ..., T_n))$$
for all $T_j \in \A$ with $j = 2,...,n$.
\end{claim}

\begin{proof}
Using the definition of $\varphi$ we have
\begin{eqnarray*}
\varphi(p_{n*}(H, T_2, ... , T_{n})) &=& p_{n*}(\varphi(H), \varphi(T_2),..., \varphi(T_{n})) \\&=& p_{n*}(\varphi(A) + \varphi(B), \varphi(T_2),..., \varphi(T_{n})) \\&=& p_{n*}(\varphi(A), \varphi(T_2), ..., \varphi(T_{n})) \\&+& p_{n*}(\varphi(B), \varphi(T_2), ..., \varphi(T_{n}))\\&=& \varphi(p_{n*}(A, T_2,..., T_{n})) \\&+& \varphi(p_{n*}(B, T_2,..., T_{n})).
\end{eqnarray*}
Similarly, we prove 
\begin{eqnarray*}
\varphi(q_{n*}(H, T_2, ..., T_n)) &=& \varphi(q_{n*}(A, T_2, ..., T_n)) \\&+& \varphi(q_{n*}(B, T_2, ..., T_n)).
\end{eqnarray*}
\end{proof}

Note that if 
\begin{eqnarray*}
&&\varphi(p_{n*}(A, T_2, ..., T_{n}))+\varphi(p_{n*}(B, T_2, ..., T_{n})) = \\&&\varphi(p_{n*}(A, T_2, ..., T_{n}) + p_{n*}(B, T_2, ..., T_{n}))
\\ and
\\&&\varphi(q_{n*}(A, T_2, ..., T_{n}))+\varphi(q_{n*}(B, T_2, ..., T_{n})) = \\&&\varphi(q_{n*}(A, T_2, ..., T_{n}) + q_{n*}(B, T_2, ..., T_{n}))
\end{eqnarray*}
then, by the injectivity of $\varphi$, we get
$$p_{n*}(H, T_2, ..., T_{n})=p_{n*}(A, T_2, ..., T_{n}) + p_{n*}(B, T_2, ..., T_{n})$$
and
$$q_{n*}(H, T_2, ..., T_{n})=q_{n*}(A, T_2, ..., T_{n}) + q_{n*}(B, T_2, ..., T_{n}).$$

\begin{claim}\label{Claim3}
Let $\A$ and $\A'$ be two  
$C^*$-algebras with identities $I_{\A}$ and $I_{\A'}$, respectively, and $P_1$ a nontrivial projection in $\A$. If $\varphi: \A \rightarrow \A'$ is a bijective unital multiplicative $*$-Lie-Jordan $n$-map then $\varphi(2^{n-1}\Re(A)) = 2^{n-1}\Re(\varphi(A))$ and $\varphi(2^{n-1}i\Im(A)) = 2^{n-1}i\Im(\varphi(A))$. 
\end{claim}

\begin{proof}
Since
$$
p_{n*} (A, I_{\A}, . . . ,I_{\A}) = 2^{n-1}i\Im(A)
$$
and
$$
q_{n*} (A, I_{\A}, . . . ,I_{\A}) = 2^{n-1}\Re(A),
$$
we obtain that
$$
\begin{aligned}
\varphi(2^{n-1}i\Im(A)) &= \varphi(p_{n*} (A, I_{\A}, . . . ,I_{\A})) = p_{n*}(\varphi(A), \varphi(I_{\A}), ... , \varphi(I_{\A})) \\&= 2^{n-1}i\Im(\varphi(A))
\end{aligned}
$$
and
$$
\begin{aligned}
\varphi(2^{n-1}\Re(A)) &= \varphi(q_{n*} (A, I_{\A}, . . . ,I_{\A})) = q_{n*}(\varphi(A), \varphi(I_{\A}), ... , \varphi(I_{\A})) \\&= 2^{n-1}\Re(\varphi(A)).
\end{aligned}
$$

\end{proof}

\section{Main theorem}

We shall prove as follows the main result of this paper:

\begin{theorem}\label{mainthm} 
Let $\A$ and $\A'$ be two  
$C^*$-algebras with identities $I_{\A}$ and $I_{\A'}$, respectively, and $P_1$, $P_2 = I_{\A} - P_1$ nontrivial projections in $\A$. 
Suppose that $\A'$ satisfies:
\begin{eqnarray*}
  \left(\dagger\right) \ \ \ \mbox{If} \ \ \  X \A' \varphi(P_i) = \left\{0\right\} \ \ \  \mbox{implies} \ \ \ X = 0.
\end{eqnarray*}
Then $\varphi: \A \rightarrow \A'$ is a bijective unital multiplicative $*$-Lie-Jordan $n$-map if and only if $\varphi$ is a
$*$-ring isomorphism.
\end{theorem}

We prove only the necessary condition because clearly sufficient condition hold true.

The following lemmas have the same hypotheses as the Theorem \ref{mainthm} and we need them for the proof of this theorem. Thus, let us consider $P_{1}$ a  nontrivial projection of $\mathfrak{A}$ and $P_2 = I_{\A} - P_1$. 

\begin{lemma}\label{lema4} For any $A_{ii} \in \A_{ii}$, $B_{ij} \in \A_{ij}$, with $i \neq j$ we have 
$$\varphi(A_{ii} + B_{ij}) = \varphi(A_{ii}) + \varphi(B_{ij})$$
and
$$\varphi(A_{ii} + B_{ji}) = \varphi(A_{ii}) + \varphi(B_{ji}).$$
\end{lemma}
\begin{proof}
We shall only prove the case $i=1$, $j=2$ because the demonstration of the other cases are similar. 
By surjectivity of $\varphi$ there exist $H = H_{11} + H_{12} + H_{21} + H_{22} \in \A$ such that $\varphi(H) = \varphi(A_{11}) + \varphi(B_{12})$. Applying the Claims \ref{Claim1} and \ref{Claim2} we have
\begin{eqnarray*}
\varphi(p_{n*}(H,P_1,...,P_1)) &=& \varphi(p_{n*}(A_{11},P_1,...,P_1)) + \varphi(p_{n*}(B_{12},P_1,...,P_1)) \\&=& \varphi(p_{n*}(A_{11},P_1,...,P_1)). 
\end{eqnarray*}
Since $\varphi$ is injective, we get $p_{n*}(H,P_1,...,P_1) = p_{n*}(A_{11},P_1,...,P_1)$, that is, $2^{n-2}(H_{11} - H^{*}_{11}) + H_{21} - H^{*}_{21} = 2^{n-2}(A_{11} - A^{*}_{11})$. Thus $\Im(H_{11}) = \Im(A_{11})$ and $\Im(H_{21}) = 0$. By similar way we have
\begin{eqnarray*}
\varphi(q_{n*}(H,P_1,...,P_1)) &=& \varphi(q_{n*}(A_{11},P_1,...,P_1)) + \varphi(q_{n*}(B_{12},P_1,...,P_1)) \\&=& \varphi(q_{n*}(A_{11},P_1,...,P_1)). 
\end{eqnarray*}
Again by injectivity of $\varphi$ follows that $q_{n*}(H,P_1,...,P_1)= q_{n*}(A_{11},P_1,...,P_1)$, that is, $2^{n-2}(H_{11} + H^{*}_{11}) + H_{21} + H^{*}_{21} = 2^{n-2}(A_{11} + A^{*}_{11})$. 
Hence $\Re(H_{11}) = \Re(A_{11})$ and  $\Re(H_{21}) = 0$ that implies $H_{11} = A_{11}$ and $H_{21} = 0$.
Finally, applying again Claims \ref{Claim1} and \ref{Claim2} for $P_2$ we have
\begin{eqnarray*}
\varphi(p_{n*}(H,P_2,...,P_2)) &=& \varphi(p_{n*}(A_{11},P_2,...,P_2)) + \varphi(p_{n*}(B_{12},P_2,...,P_2)) \\&=& \varphi(p_{n*}(B_{12},P_2,...,P_2)). 
\end{eqnarray*}
Since $\varphi$ is injective, we get $p_{n*}(H,P_2,...,P_2) = p_{n*}(B_{12},P_2,...,P_2)$, that is, $2^{n-2}(H_{22} - H^{*}_{22}) + H_{12} - H^{*}_{12} = B_{12} - B^{*}_{12}$. Thus $\Im(H_{12}) = \Im(B_{12})$ and $\Im(H_{22}) = 0$. Similarly, we obtain $q_{n*}(H,P_2,...,P_2) = q_{n*}(B_{12},P_2,...,P_2)$, that is, $2^{n-2}(H_{22} + H^{*}_{22}) + H_{12} + H^{*}_{12} = B_{12} + B^{*}_{12}$. Hence $\Re(H_{12}) = \Re(B_{12})$ and $\Re(H_{22}) = 0$ which implies $H_{12} = B_{12}$, $H_{22} = 0$ and
$H=A_{11} + B_{12}$. 

\end{proof}

\begin{lemma}\label{lema5}
For any $A_{12} \in \A_{12}$ and $B_{21} \in \A_{21}$, we have $\varphi(A_{12} + B_{21}) = \varphi(A_{12}) + \varphi(B_{21})$.
\end{lemma}

\begin{proof}
Let $H = H_{11} + H_{12} + H_{21} + H_{22} \in \A$ be such that $\varphi(H) = \varphi(A_{12}) + \varphi(B_{21})$.  Applying the Claims \ref{Claim1} and \ref{Claim2} we have
\begin{eqnarray*}
\varphi(p_{n*}(H,P_1,...,P_1)) &=& \varphi(p_{n*}(A_{12},P_1,...,P_1)) + \varphi(p_{n*}(B_{21},P_1,...,P_1)) \\&=& \varphi(p_{n*}(B_{21},P_1,...,P_1)). 
\end{eqnarray*}  
Since $\varphi$ is injective, we get $p_{n*}(H,P_1,...,P_1) = p_{n*}(B_{21},P_1,...,P_1)$, that is, $2^{n-2}(H_{11} - H^{*}_{11}) + H_{21} - H^{*}_{21} = B_{21} - B^{*}_{21}$. Thus $\Im(H_{11}) = 0$ and $\Im(H_{21}) = \Im(B_{21})$.
Similarly, we can obtain $\Re(H_{11}) = 0$ and $\Re(H_{21}) = \Re(B_{21})$ by applying 
\begin{eqnarray*}
\varphi(q_{n*}(H,P_1,...,P_1)) &=& \varphi(q_{n*}(A_{12},P_1,...,P_1)) + \varphi(q_{n*}(B_{21},P_1,...,P_1)) \\&=& \varphi(q_{n*}(B_{21},P_1,...,P_1)). 
\end{eqnarray*}
Now, we apply
\begin{eqnarray*}
\varphi(p_{n*}(H,P_2,...,P_2)) &=& \varphi(p_{n*}(A_{12},P_2,...,P_2)) + \varphi(p_{n*}(B_{21},P_2,...,P_2)) \\&=& \varphi(p_{n*}(A_{12},P_2,...,P_2)). 
\end{eqnarray*} 
Again since $\varphi$ is injective, we get $p_{n*}(H,P_2,...,P_2) = p_{n*}(A_{12},P_2,...,P_2)$, that is, $2^{n-2}(H_{22} - H^{*}_{22}) + H_{12} - H^{*}_{12} = A_{12} - A^{*}_{12}$. Thus $\Im(H_{22}) = 0$ and $\Im(H_{12}) = \Im(A_{12})$.
And the similar way we get $\Re(H_{22}) = 0$ and $\Re(H_{12}) = \Re(A_{12})$ by applying
\begin{eqnarray*}
\varphi(q_{n*}(H,P_2,...,P_2)) &=& \varphi(q_{n*}(A_{12},P_2,...,P_2)) + \varphi(q_{n*}(B_{21},P_2,...,P_2)) \\&=& \varphi(q_{n*}(A_{12},P_2,...,P_2)). 
\end{eqnarray*}
Therefore $H = A_{12} + B_{21}$ and the Lemma hold true.
 
\end{proof}

\begin{lemma}\label{lema6}
For any $A_{ij}, B_{ij} \in \A_{ij}$ with $i \neq j$, we have $\varphi(A_{ij} + B_{ij}) = \varphi(A_{ij}) + \varphi(B_{ij})$.
\end{lemma}

\begin{proof}
Let $H = H_{11} + H_{12} + H_{21} + H_{22} \in \A$ be such that $\varphi(H) = \varphi(A_{ij}) + \varphi(B_{ij})$.  Applying the Claims \ref{Claim1} and \ref{Claim2} we have
\begin{eqnarray*}
\varphi(p_{n*}(H,P_i,...,P_i)) &=& \varphi(p_{n*}(A_{ij},P_i,...,P_i)) + \varphi(p_{n*}(B_{ij},P_i,...,P_i)) \\&=& \varphi(0) \\&=& 0. 
\end{eqnarray*}  
Hence $H_{ii} + H_{ji} - H^{*}_{ii} - H^{*}_{ji} = 0$ and $\Im(H_{ii}) = \Im(H_{ji}) = 0$. Similarly, using
\begin{eqnarray*}
\varphi(q_{n*}(H,P_i,...,P_i)) &=& \varphi(q_{n*}(A_{ij},P_i,...,P_i)) + \varphi(q_{n*}(B_{ij},P_i,...,P_i)) \\&=& \varphi(0) \\&=& 0, 
\end{eqnarray*}
we can obtain $\Re(H_{ii}) = \Re(H_{ji}) = 0$. Therefore $H_{ii} = H_{ji} = 0$.
Now we apply Claim \ref{Claim2} and Lemma \ref{lema5} to get
\begin{eqnarray*}
\varphi(p_{n*}(H,P_j,...,P_j)) &=&  \varphi(p_{n*}(A_{ij},P_j,...,P_j)) + \varphi(p_{n*}(B_{ij},P_j,...,P_j)) \\&=& \varphi(A_{ij} - A^{*}_{ij}) + \varphi(B_{ij} - B^{*}_{ij}) \\&=& \varphi(2\Im(A_{ij})) + \varphi(2\Im(B^{*}_{ij})) \\&=& \varphi(2\Im(A_{ij}) + 2\Im(B^{*}_{ij})) \\&=& \varphi(2\Im(A_{ij}) + 2\Im(B_{ij})).
\end{eqnarray*}
So $H_{ij} + H_{jj} - H^{*}_{ij} - H^{*}_{jj} = 2(\Im(A_{ij}) + \Im(B_{ij}))$. Hence $\Im(H_{ij}) = \Im(A_{ij}) + \Im(B_{ij})$ and $\Im(H_{jj}) = 0$.
Similarly, we can obtain $\Re(H_{ij}) = \Re(A_{ij}) + \Re(B_{ij})$ and $\Re(H_{jj}) = 0$. Therefore $H_{ij} = A_{ij} + B_{ij}$.
\end{proof}

\begin{lemma}\label{lema8}
For any $A_{11} \in \A_{11}$, $B_{12} \in \A_{12}$, $C_{21} \in \A_{21}$, $D_{22} \in \A_{22}$, we have 
$$\varphi(A_{11} + B_{12} + C_{21} + D_{22}) = \varphi(A_{11}) + \varphi(B_{12}) + \varphi(C_{21}) + \varphi(D_{22}).$$
\end{lemma}
\begin{proof}
By surjectivity of $\varphi$ there exist $H = H_{11} + H_{12} + H_{21} + H_{22} \in \A$ such that $\varphi(H) = \varphi(A_{11}) + \varphi(B_{12}) + \varphi(C_{21}) + \varphi(D_{22})$. Applying the Claims \ref{Claim1}, \ref{Claim2} and Lemma \ref{lema4} we have
\begin{eqnarray*}
\varphi(p_{n*}(H,P_1,...,P_1)) &=& \varphi(p_{n*}(A_{11},P_1,...,P_1)) + \varphi(p_{n*}(B_{12},P_1,...,P_1)) \\&+& \varphi(p_{n*}(C_{21},P_1,...,P_1)) + \varphi(p_{n*}(D_{22},P_1,...,P_1)) \\&=& \varphi(2^{n-2}(A_{11} - A^{*}_{11})) + \varphi(C_{21} - C^{*}_{21}) \\&=& \varphi(2^{n-2}(A_{11} - A^{*}_{11}) + (C_{21} - C^{*}_{21}))
\end{eqnarray*}
Since $\varphi$ is injective, we get $2^{n-2}(H_{11} - H^{*}_{11}) + H_{21} - H^{*}_{21} = 2^{n-2}(A_{11} - A^{*}_{11}) + C_{21} - C^{*}_{21}$. Thus,
$\Im(H_{11}) = \Im(A_{11})$ and $\Im(H_{21}) = \Im(C_{21})$. We can obtain $\Re(H_{11}) = \Re(A_{11})$ and $\Re(H_{21}) = \Re(C_{21})$ by using 
\begin{eqnarray*}
\varphi(q_{n*}(H,P_1,...,P_1)) &=& \varphi(q_{n*}(A_{11},P_1,...,P_1)) + \varphi(q_{n*}(B_{12},P_1,...,P_1)) \\&+& \varphi(q_{n*}(C_{21},P_1,...,P_1)) + \varphi(q_{n*}(D_{22},P_1,...,P_1)).
\end{eqnarray*} 
Now for $P_2$ and Lemma \ref{lema4} it follows that
\begin{eqnarray*}
\varphi(p_{n*}(H,P_2,...,P_2)) &=& \varphi(p_{n*}(A_{11},P_2,...,P_2)) + \varphi(p_{n*}(B_{12},P_2,...,P_2)) \\&+& \varphi(p_{n*}(C_{21},P_2,...,P_2)) + \varphi(p_{n*}(D_{22},P_2,...,P_2)) \\&=& \varphi(B_{12} - B^{*}_{12}) + \varphi(2^{n-2}(D_{22} - D^{*}_{22})) \\&=& \varphi(B_{12} - B^{*}_{12} + 2^{n-2}(D_{22} - D^{*}_{22})),
\end{eqnarray*}
which implies $2^{n-2}(H_{22} - H^{*}_{22}) + H_{12} - H^{*}_{12} = B_{12} - B^{*}_{12} + 2^{n-2}(D_{22} - D^{*}_{22})$. Therefore $\Im(H_{12}) = \Im(B_{12})$ and $\Im(H_{22}) = \Im(D_{22})$. Similarly, we can obtain  $\Re(H_{12}) = \Re(B_{12})$ and $\Re(H_{22}) = \Re(D_{22})$ by using
\begin{eqnarray*}
\varphi(q_{n*}(H,P_2,...,P_2)) &=& \varphi(q_{n*}(A_{11},P_2,...,P_2)) + \varphi(q_{n*}(B_{12},P_2,...,P_2)) \\&+& \varphi(q_{n*}(C_{21},P_2,...,P_2)) + \varphi(q_{n*}(D_{22},P_2,...,P_2)).
\end{eqnarray*}
Thus, we proved that $H = A_{11} + B_{12} + C_{21} + D_{22}$.
\end{proof}



\begin{lemma}\label{lema9}
We have $\varphi(AP_i)= \varphi(A)\varphi(P_i)$ and $\varphi(P_i A) = \varphi(P_i)\varphi(A)$ for $i \in \left\{1,2\right\}$.
\end{lemma}

\begin{proof}
Since
$$
q_{n*}(A,P_i, I_\A, ..., I_\A) = 2^{n-2}(AP_i + P_iA^*)
$$
and
$$
p_{n*}(A,P_i, I_\A, ..., I_\A) = 2^{n-2}(AP_i - P_iA^*)
$$
we have
\begin{eqnarray*}
\varphi(2^{n-2}(AP_i + P_iA^*))&=&\varphi(q_{n*}(A,P_i, I_\A, ..., I_\A)) \\&=& q_{n*}(\varphi(A),\varphi(P_i), \varphi(I_\A), ..., \varphi(I_\A)) \\&=& 2^{n-2}(\varphi(A)\varphi(P_i) + \varphi(P_i)\varphi(A)^{*})
\end{eqnarray*}
and
\begin{eqnarray*}
\varphi(2^{n-2}(AP_i - P_iA^*))&=&\varphi(p_{n*}(A,P_i, I_\A, ..., I_\A)) \\&=& p_{n*}(\varphi(A),\varphi(P_i), \varphi(I_\A), ..., \varphi(I_\A)) \\&=& 2^{n-2}(\varphi(A)\varphi(P_i) - \varphi(P_i)\varphi(A)^{*}).
\end{eqnarray*}
Adding these two equations we obtain

\begin{eqnarray*}
2^{n-1}\varphi(A)\varphi(P_i) &=& \varphi(2^{n-2}(AP_i + P_iA^*)) + \varphi(2^{n-2}(AP_i - P_iA^*)) \\&=& \varphi(2^{n-2}(A_{ii} + A_{ji} + A^*_{ii} + A^*_{ji})) \\&+&
\varphi(2^{n-2}(A_{ii} + A_{ji} - A^*_{ii} - A^*_{ji})) \\&=& \varphi(2^{n-2}(A_{ii} + A^{*}_{ii})) + \varphi(2^{n-2}(A_{ji} + A^{*}_{ji})) \\&+& \varphi(2^{n-2}(A_{ii} - A^{*}_{ii})) + \varphi(2^{n-2}(A_{ji} - A^{*}_{ji})) \\&=& \varphi(2^{n-1}\Re(A_{ii})) + \varphi(2^{n-1}i\Im(A_{ii})) \\&+& \varphi(2^{n-2}A_{ji})+ \varphi(2^{n-2}A^{*}_{ji}) + \varphi(2^{n-2}A_{ji}) - \varphi(2^{n-2}A^{*}_{ji}) \\&=& 2^{n-1}\Re(\varphi(A_{ii})) + 2^{n-1}i\Im(\varphi(A_{ii})) \\&+& \varphi(2^{n-1}A_{ji}) \\&=& 2^{n-1} \varphi(A_{ii} + A_{ji}) \\&=& 2^{n-1} \varphi(AP_{i}).
\end{eqnarray*}
Similarly we get $\varphi(P_i A) = \varphi(P_i)\varphi(A)$.

\end{proof}

Observe that the previous lemma ensure that if $P_i$ is a projection, then $\varphi(P_i)$ and $\varphi^{-1}(P_i)$ are projection too. Indeed,
$$
\varphi(P_i)^2 = \varphi(P_i)\varphi(P_i) = \varphi(P_iP_i) = \varphi(P_i).
$$ 
Since $\varphi^{-1}$ has the same properties of $\varphi$, we conclude that $\varphi^{-1}$ also preserve projections.

Before we prove the next Lemma note that, by Lemma \ref{lema8} and $\varphi(P_i) = Q_i$ for $i \in \left\{1,2\right\}$,  $\A' = \A'_{11} + \A'_{12} + \A'_{21} + \A'_{22}$ where $\A'_{ij} = Q_i \A' Q_j$, $i,j \in \left\{1,2\right\}$.

\begin{lemma}\label{lema11}
$\varphi(\A_{ij}) = \A'_{ij}$.
\end{lemma}

\begin{proof}
Firstly, observe that $\varphi(\A_{ij}) = \varphi(P_i)\varphi(\A)\varphi(P_j) = \varphi(P_i)\A'\varphi(P_j) \subseteq A'_{ij}$ hold true because $\varphi$ is surjective and by Lemma \ref{lema9}.
Now, let be $B_{ij} \in \A'_{ij}$. By surjectivity of $\varphi$ there exist $A \in \A$ such that $\varphi(A) = B_{ij}$. So by Lemma \ref{lema9} $B_{ij} = Q_i B_{ij} Q_j = \varphi(P_i) \varphi(A) \varphi(P_j) = \varphi(P_i A P_j)$.
   
\end{proof}

\begin{lemma}\label{lema12}
For every $A_{ii}, B_{ii} \in \A_{ii}$, we have $\varphi(A_{ii} + B_{ii}) = \varphi(A_{ii}) + \varphi(B_{ii})$ for $i \in \left\{1,2\right\}.$

\end{lemma}

\begin{proof}
By Lemma \ref{lema6}, \ref{lema9} and \ref{lema11} we have
\begin{eqnarray*}
\varphi(P_iA + P_iB)\varphi(H_{ji}) &=& \varphi(P_iA + P_iB)\varphi(P_j)\varphi(HP_i) \\&=& \varphi(P_iAP_j + P_iBP_j)\varphi(HP_i) \\&=& (\varphi(P_iAP_j) + \varphi(P_iBP_j))\varphi(HP_i) \\&=& (\varphi(P_iA) + \varphi(P_iB))\varphi(P_j)\varphi(HP_i) \\&=& (\varphi(P_iA) + \varphi(P_iB))\varphi(H_{ji}).  
\end{eqnarray*}
Hence $(\varphi(P_iA + P_iB) - (\varphi(P_iA) + \varphi(P_iB)))Q_j\A'Q_i = 0$ and by assumption ($\dagger$) of Theorem \ref{mainthm} we get 
$\varphi(P_iA + P_iB) = \varphi(P_iA) + \varphi(P_iB)$. 
Again using Lemma \ref{lema9} we obtain $\varphi(P_iAP_i + P_iBP_i) = \varphi(P_iA + P_iB)\varphi(P_i) = (\varphi(P_iA) + \varphi(P_iB))\varphi(P_i) = \varphi(P_iAP_i) + \varphi(P_iBP_i)$.
\end{proof}

\vspace{0,5cm}

We are ready to prove our Theorem \ref{mainthm}:

\vspace{0,5cm}

\noindent Proof of Theorem. Let be $A, B \in \A$ with $A= A_{11} + A_{12} + A_{21} + A_{22}$ and $B= B_{11} + B_{12} + B_{21} + B_{22}$. By previous Lemmas we obtain
\begin{eqnarray*}
\varphi(A + B) &=& \varphi(A_{11} + A_{12} + A_{21} + A_{22} + B_{11} + B_{12} + B_{21} + B_{22})
\\&=& \varphi((A_{11} + B_{11}) + (A_{12} + B_{12}) + (A_{21}+ B_{21}) + (A_{22} + B_{22}))
\\&=& \varphi(A_{11}) +\varphi( B_{11}) + \varphi(A_{12}) + \varphi(B_{12}) + \varphi(A_{21})
\\&+& \varphi(B_{21}) +\varphi(A_{22}) + \varphi(B_{22}) 
\\&=& (\varphi(A_{11}) +\varphi(A_{12}) +\varphi(A_{21})+\varphi(A_{22})) 
\\&+& (\varphi( B_{11}) + \varphi( B_{12}) + \varphi( B_{21}) + \varphi( B_{22}))
\\&=& \varphi(A_{11} + A_{12} + A_{21} + A_{22}) + \varphi(B_{11} + B_{12} + B_{21} + B_{22}) 
\\&=& \varphi(A) + \varphi(B).
\end{eqnarray*}

Now since $\varphi$ is additive it follows that 
$$\varphi(A + A^{*}) = \varphi(A) + \varphi(A^{*}).$$
Using the definition of $\varphi$ for $I_{\A}$ we have
\begin{eqnarray*}
2^{n-2}\varphi(A+A^{*}) = \varphi(2^{n-2}(A+A^{*})) &=& \varphi(q_{n*}(A,I_{\A}, ... , I_{\A})) \\&=& q_{n*}(\varphi(A),\varphi(I_{\A}), ... , \varphi(I_{\A})) \\&=& 2^{n-2}(\varphi(A) + \varphi(A)^{*}).
\end{eqnarray*}
We conclude that $2^{n-2}(\varphi(A) + \varphi(A^{*})) = 2^{n-2}\varphi(A + A^{*}) = 2^{n-2}(\varphi(A) + \varphi(A)^{*}).$ Thus $\varphi(A^{*}) = \varphi(A)^{*}$.

Observe that
\begin{eqnarray*}
2^{n-2}\varphi(AB - BA^{*})&=&\varphi(2^{n-2}(AB - BA^{*}))=\varphi(p_{n*} (I_{\A}, . . . ,I_{\A}, A, B)) \\&=& p_{n*}(\varphi(I_{\A}), ... , \varphi(I_{\A}), \varphi(A), \varphi(B))\\&=& 2^{n-2}(\varphi(A)\varphi(B) - \varphi(B)\varphi(A)^{*})
\end{eqnarray*}
and
\begin{eqnarray*}
2^{n-2}\varphi(AB + BA^{*})&=&\varphi(2^{n-2}(AB + BA^{*}))=\varphi(q_{n*} (I_{\A}, . . . ,I_{\A}, A, B)) \\&=& q_{n*}(\varphi(I_{\A}), ... , \varphi(I_{\A}), \varphi(A), \varphi(B))\\&=& 2^{n-2}(\varphi(A)\varphi(B) + \varphi(B)\varphi(A)^{*}),
\end{eqnarray*}
for all $A, B \in \A$. 
So we obtain that $\varphi(AB) = \varphi(A)\varphi(B)$ and this finishes the proof of Theorem \ref{mainthm}.

Note that prime $C^*$-algebras satisfy $(\dagger)$ hence we have the follow result

\begin{corollary} 
Let $\A$ be $C^*$-algebra and $\A'$ be prime $C^*$-algebra with identities $I_{\A}$ and $I_{\A'}$ respectively and $P_1$, $P_2 = I_{\A} - P_1$ nontrivial projections in $\A$. Then $\varphi: \A \rightarrow \A'$ is a bijective unital multiplicative $*$-Lie-Jordan $n$-map if only if $\varphi$ is a $*$-ring isomorphism.
\end{corollary}

\begin{remark}
Observe that if $\varphi$ satisfies only
$$\varphi(p_{n*}(A, P, ..., P)) = p_{n*}(\varphi(A), \varphi(P), ..., \varphi(P))$$
and
$$\varphi(q_{n*}(A, P, ..., P)) = q_{n*}(\varphi(A), \varphi(P), ..., \varphi(P)),$$
for all $A \in \A$ and $P \in \left\{P_1, I_{\A} - P_1\right\}$ then

$$\varphi(p_{n*}(A, I_{\A}, ..., I_{\A})) = p_{n*}(\varphi(A), \varphi(I_{\A}), ..., \varphi(I_{\A}))$$
and
$$\varphi(q_{n*}(A, I_{\A}, ..., I_{\A})) = q_{n*}(\varphi(A), \varphi(I_{\A}), ..., \varphi(I_{\A})),$$
respectively, just by using linearity of $p_{n*}$ and $q_{n*}$ to $I_{\A} = P_1 + P_2$.
\end{remark}

In view of this follows the result
\begin{theorem}
Let $\A$ and $\A'$ be two  
$C^*$-algebras with identities $I_{\A}$ and $I_{\A'}$, respectively, and $P_1$ and $P_2 = I_{\A} - P_1$ nontrivial projections in $\A$. 
Suppose that $\A'$ satisfies:
\begin{eqnarray*}
  \left(\dagger\right) \ \ \ \mbox{If} \ \ \ X \A' \varphi(P_i) = \left\{0\right\} \ \ \  \mbox{implies} \ \ \ X = 0.
\end{eqnarray*}
If $\varphi: \A \rightarrow \A'$ is a bijective unital map which satisfies
\begin{eqnarray*}
 &&\varphi(p_{n*}(A, P, ..., P)) = p_{n*}(\varphi(A), \varphi(P), ..., \varphi(P))
\\&and&
\\&&\varphi(q_{n*}(A, P, ..., P)) = q_{n*}(\varphi(A), \varphi(P), ..., \varphi(P)),
\end{eqnarray*}
 for all $A \in \A$ and $P \in \left\{P_1, P_2\right\}$ then $\varphi$ is $*$-additive.
\end{theorem}

As an application of Theorem \ref{mainthm} on von Neumann algebras, we have the following corollary 

\begin{corollary}
Let $\mathcal{M}_{\A}$ be a von Neumann algebra relative $C^{*}$-algebra $\A$ without central summands of type $I_1$. Then $\varphi: \mathcal{M}_{\A} \rightarrow \mathcal{M}_{\A}$ is a bijective unital multiplicative $*$-Lie-Jordan $n$-map if and only if $\varphi$ is a $*$-ring isomorphism.
\end{corollary}
\begin{proof}
Let be $\mathcal{M}_{\A}$ the von Neumann algebra relative alternative $C^{*}$-algebra $\A$.
It is shown in \cite{Bai} that, if a von Neumann algebra $\mathcal{M}$ has no central summands of type $I_1$  $\left( = \mbox{central abelian projection} \right)$, then $\mathcal{M}$ satisfies the following assumption: 
\begin{itemize}
\item $X \mathcal{M}_{\A}P_i = \left\{0\right\} \Rightarrow X = 0$.
\end{itemize}
Now, by Theorem \ref{mainthm}, the corollary is true.
\end{proof}

To conclude this paper, we conjecture that on mild conditions in $C^{*}$-algebras every multiplicative $*$-Lie-type maps and multiplicative $*$-Jordan-type maps are $*$-ring isomorphism.

\end{document}